\documentclass[12pt,reqno]{amsart}
\usepackage{amsfonts}
\usepackage{amssymb}
\usepackage{amsmath}
\usepackage{latexsym}
\usepackage{amsthm}
\usepackage{epsfig}
\usepackage{color}
\usepackage{amsfonts,amssymb,mathrsfs,amscd}
\usepackage{comment}

\newtheorem{theorem}{Theorem}[section]
\newtheorem{corollary}[theorem]{Corollary}

\newtheorem{lemma}[theorem]{Lemma}
\newtheorem{proposition}[theorem]{Proposition}

\theoremstyle{definition}
\newtheorem{definition}[theorem]{Definition}
\newtheorem{remark}[theorem]{Remark}

\numberwithin{equation}{section}

\newcommand{\func}[1]{\operatorname{#1}}

\begin{document}

\title[ Strong attractors
for the damped-driven Euler system] {Strong trajectory and global
$\mathbf{W^{1,p}}$-attractors for the damped-driven Euler system in
$\mathbb  R^2$}
\author[V.V. Chepyzhov, A.A. Ilyin, and S.V Zelik]
{Vladimir Chepyzhov${}^{1,3}$, Alexei Ilyin ${}^{1,2}$ and Sergey
Zelik${}^{4}$}

\begin{abstract}
We consider  the damped and driven two-dimensional
 Euler equations in the plane with weak solutions
 having finite energy and enstrophy. We show that these
 (possibly non-unique) solutions satisfy the energy and enstrophy
 equality. It is shown that this system has a strong global and
 a strong trajectory attractor in the Sobolev space $H^1$.
A similar result on the strong attraction  holds in
the spaces $H^1\cap\{u:\ \|\func{curl} u\|_{L^p}<\infty\}$
for $p\ge2$.
\end{abstract}


\thanks{The research of V.~Chepyzhov and A.~Ilyin was carried out in the
Institute for Information Transmission Problems, Russian Academy of
Science, at the expense of the Russian Science Foundation (project
14-50-00150). }

\address{
 \newline ${}^1$ Institute for Information Transmission Problems, Moscow 127994, Russia,
 \newline ${}^2$ Keldysh Institute of Applied Mathematics, Moscow 125047, Russia,
 \newline ${}^3$ National Research University Higher School of Economics, Moscow 101000,
 Russia,
 \newline ${}^4$ University of Surrey, Department of Mathematics, Guildford, GU2 7XH, UK.}

\email{chep@\,iitp.ru}

\email{ilyin@\,keldysh.ru}

\email{s.zelik@\,surrey.ac.uk}

\maketitle

\section{Introduction
\label{Sec0}}

In this work we study the following  dissipative  Euler equations
\begin{equation}\label{DDE}
\begin{cases}
\partial_t u+(u,\nabla)u+\nabla p+ru=g,\\
u\big|_{t=0}=u_0,\ \ \func{div} u=0
\end{cases}
\end{equation}
in the whole plane $x\in\mathbb{R}^2$. This system with linear
damping term $ru$, where $r>0$ is the Rayleigh or  Ekman friction
coefficient, is relevant in large scale geophysical
models~\cite{Ped}. See also~\cite{BP}, where this hydrodynamic
system is derived under certain assumptions from the Boltzmann
equations.

This system has attracted considerable attention over the last years.
Our interest  in this work is in the time dependent problem,
but for the moment we observe that
the regularity and uniqueness  of the
stationary solutions of~\eqref{DDE} in a bounded domain
in $\mathbb{R}^2$ were studied in Sobolev spaces \cite{BCT}, \cite{Sau}.
In the framework of the Lagrangian approach the
stability of  stationary solutions in H\"older spaces
was studied in~\cite{W}.

Closely related to \eqref{DDE} is the Navier--Stokes
perturbation of it
\begin{equation}\label{DDNS}
\begin{cases}
\partial_t u+(u,\nabla)u+\nabla p+ru=\nu \Delta u+g,\\
u\big|_{t=0}=u_0,\ \ \func{div} u=0,
\end{cases}
\end{equation}
which is studied in the vanishing viscosity limit $\nu\to0$.

We have to distinguish two fundamentally different phase spaces in
which systems \eqref{DDE} and \eqref{DDNS} are studied in case of
unbounded domains $\Omega\subseteq\mathbb{R}^2$. Namely, the usual
Sobolev spaces (so that, for instance in the case of $L^2$, we are
dealing with finite energy solutions) and the uniformly local
spaces in which the Sobolev norms on balls of arbitrary fixed
radius are uniformly bounded with respect to all translations.

In the case of periodic boundary conditions the estimate
of the fractal dimension of the global attractor for~\eqref{DDNS}
was obtained in~\cite{IMT},
where it was shown that the corresponding dynamical system
possesses a global attractor $\mathcal{A}$ (in $L^2$) with fractal dimension
satisfying the  estimate
\begin{equation}\label{min}
\dim_f\mathcal{A}\le
\frac38\,\frac{\|\nabla g\|^2}{\nu r^3}.
\end{equation}
 It was also shown in \cite{IMT} that if $r>0$ is fixed, then the  rate
of growth of the dimension estimate, as $\nu\to0$,   is sharp, and
the upper bounds were supplemented with a lower bound of the order
$1/\nu$. See also~\cite{IT} for the estimates of the same order for
number of the degrees of freedom expressed in terms of various
finite dimensional projections.

Since estimate~\eqref{min} does not explicitly contain the length scale,
it is reasonable to expect that this estimate survives in
$\mathbb{R}^2$ in the phase space of finite energy solutions.
It has recently been shown in~\cite{IPZ} that this is indeed the case.
Furthermore, due to the scale invariance of $\mathbb{R}^2$, the estimate
was imbedded in the family of estimates depending on the norm
of the right-hand side $g$ in the whole family of homogeneous
Sobolev spaces:
$$
\dim_f\mathcal A\le \frac{1-s^2}{64\sqrt{3}}
\left(\frac{1+|s|}{1-|s|}\right)^{|s|}\frac1{r^{2+s}\,\nu^{2-s}}\|g\|^2_{\dot H^{s}},
\quad s\in[-1,1].
$$
In particular, for $s=1$ we obtain
$$
\dim_f\mathcal A\le \frac{1}{16\sqrt{3}}
\frac{\|\nabla g\|^2}{\nu\,r^{3}},
$$
which up to a constant agrees with \eqref{min}.

In the case of uniformly local spaces (where the energy and higher
Sobolev norms in the whole $\mathbb{R}^2$ are infinite) one of the main
issues is the proof of the dissipative estimate.
For the uniformly local spaces in the viscous case $\nu>0$ the global attractors
for \eqref{DDNS}  in the strong topology
were constructed in~\cite{Ze13}, see also~\cite{Ze07,Ze8} for similar results
in channel-like domains. In the inviscid case the strong attractor
for \eqref{DDE} in the uniformly local $H^1$ space was recently
constructed in \cite{CZ15}.

For the  damped-driven Euler equations~\eqref{DDE} on the torus in
the phase space of finite energy (and enstrophy) solutions (where
the solutions are not necessarily unique) the weak $H^1$ attractors
were constructed in~\cite{IlyinEU} using the approach of
\cite{BV_MS} (see also~\cite{BF}). The corresponding weak $H^1$
trajectory attractors were constructed in~\cite{CVrj} and
in~\cite{Ch15} for the non-autonomous case of \eqref{DDE}.

A method for proving the strong attraction and compactness for the attractors
of equations in unbounded or non-smooth domains was proposed
and developed in \cite{B,Gh,MRW,Rosa}. The method essentially uses  the
energy equality for the corresponding equation for the proof of the asymptotic
compactness of the solution semigroup. In our case this method was applied
in~\cite{CVZ} for the proof of the existence of the strong $H^1$ trajectory
attractor for~\eqref{DDE} on the 2D torus. We point that
 the uniqueness of the solutions lying on the attractor
 for system~\eqref{DDE} with $\func{curl}g\in L^\infty$
was used in~\eqref{DDE} for the proof of the enstrophy equality.

In this work we obtain the energy (more precisely, enstrophy) equality in the Sobolev spaces
$W^{1,p}$, $2\le p<\infty$ without the assumption on $g$
that guarantees the uniqueness on the attractor. Instead we
use the fact that in the 2D case the vorticity satisfies the transport equation,
and the corresponding enstrophy equality directly follows from the results
of~\cite{DiP-Lio}.

In Section~\ref{Sec1} we give the definition of the weak solution
and prove that a solution exists and satisfies the energy equality.
However, the regularity   of it is insufficient for the
straightforward  proof of the enstrophy  equality.

In Section~\ref{Sec2} we show that the vorticity
$\omega=\func{curl}u$ satisfies in the weak sense the transport
equation
$$
\partial _{t}\omega +(u,\nabla )\omega +r\omega =\func{curl}g(x),
$$
where  $u\in L^\infty(0,T;H^1)$ is a weak solution of~\eqref{DDE}.
The results of~\cite{DiP-Lio} then immediately imply the required
enstrophy equality.

In Section~\ref{Sec3} we construct the weak trajectory attractor
for system~\eqref{DDE} and describe its structure.

In Section~\ref{Sec4} we recall the definition of the global
attractor for a~system with possibly non-unique solutions. Based on
the enstrophy equality obtained in Section~\ref{Sec2} we  prove the
existence of the global attractor that is compact in $H^1$, and
then show that the weak trajectory attractor constructed in
Section~\ref{Sec3} is, in fact, the strong trajectory attractor.

In Section~\ref{Sec5} we prove further regularity of the attractor
under the assumption that $\func{curl}g\in  L^p$, $p\ge2$. Namely,
we prove the compactness and attraction in the norm of the space
$H^1\cap\mathcal{E}_p$, where
$$
\mathcal{E}_p:=\{u,\ \func{div}u=0,\ \|u\|_{\mathcal{E}_p}:=
\|u\|_{L^2}+\|\func{curl}u\|_{L^p}<\infty\}.
$$

Finally, in Section~\ref{Sec6} we give a variant of the Yudovich's
proof \cite{Yud2,Yud3} of the uniqueness  if $\func{curl}g\in
L^\infty$ (see also~\cite{CVZ}).

\section{Weak solutions for damped Euler system in $\mathbb{R}^{2}$
\label{Sec1}}

We study the following 2D damped-driven Euler system:
\begin{equation}
\left\{
\begin{array}{l}
\partial _{t}u+(u,\nabla )u+\nabla p+ru=g(x), \\
\func{div}u:=\partial _{1}u^{1}+\partial _{2}u^{2}=0,\quad
x=(x_{1},x_{2})\in \mathbb{R}^{2},\quad t\geq 0.
\end{array}
\right.   \label{s1r1}
\end{equation}
 Here, $u=u(t,x)=(u^{1}(t,x),u^{2}(t,x))$ is the
unknown velocity vector field, $p=p(t,x)$ is the unknown pressure,
$g(x)=(g^{1}(x),g^{2}(x))$ is the given external force, $\nabla
=(\partial _{1},\partial _{2})$, and $\partial _{i}:=\partial
/\partial {x_{i}},\ i=1,2$. Next,
\begin{equation*}
(u,\nabla )v=u^{1}\partial _{1}v+u^{2}\partial _{2}v  \label{s1r2}
\end{equation*}
is the bilinear term, and $r>0$ is the given parameter that describes  the Ekman damping
term $ru$. Without
loss of generality  we shall assume that $\func{div}g=0$.

We study weak solutions of the system (\ref{s1r1}) with \emph{finite}
energy. We set
\begin{equation*}
\mathcal{H}:=\{v\in (L_{2}(\Omega ))^{2},\ \func{div}v=0\}.
\end{equation*}
The norm in the space $\mathcal{H}$ is denoted by
$\Vert \cdot\Vert $. We shall use  standard  notation for the Sobolev spaces
$W^{\ell ,p}=W^{\ell ,p}(\mathbb{R}^{2})$. As usual,
$H^{s}=W^{s,2}$. Along with the space $\mathcal{H}=\mathcal{H}^{0}$
we shall use the standard scale of spaces $\mathcal{H}^{s}\subset
(H^{s})^{2}$. In particular, $\mathcal{H}^{1}=$
$\{u\in (H^{1})^{2},\ \func{div}u=0\}$. We clearly have that
$\mathcal{H}^{-s}=(\mathcal{H}^{s})^{\ast }$ is the dual space to $\mathcal{H}^{s}$ for
$s>0$. The norm in $\mathcal{H}^{s}$ is denoted by $\| \cdot \|_{s}$.

In two (and three) dimensions we have for the vector Laplacian
\begin{equation}\label{invLapl}
\Delta u= (\Delta u^1,\Delta u^2)=\func{curl} \func{curl}u-\nabla\func{div}u,
\end{equation}
so that integrating by parts we obtain for $ u\in \mathcal{H}^{1}$
\begin{equation}
\Vert u\Vert _{1}^{2}=\|u\|^2+\|\nabla u\|^2=\Vert u\Vert ^{2}+\Vert \func{div}u\Vert ^{2}+\Vert
\func{curl}u\Vert ^{2}=\Vert u\Vert ^{2}+\Vert \func{curl}u\Vert ^{2}.
  \label{s1r3}
\end{equation}

 We shall also use the spaces $\mathcal{W}^{1 ,p}$, $1<p<\infty$, with norm
\begin{equation} \label{s1r4}
\|u\|_{\mathcal{W}^{1 ,p}}=\|u\|_{L^p}+\|\func{curl}u\|_{L^p}.
\end{equation}
Smooth divergence free vector functions
$v\in \left( C_{0}^{\infty }(\mathbb{R}^2\right) )^{2}$, $\func{div}v=0$
are dense in $\mathcal{W}^{1 ,p}$.

The functional formulation of the system~\eqref{s1r1} is similar to
that of the Navier--Stokes system \cite{TemNS}. The trilinear form
\begin{equation}  \label{s1r6}
b(u,v,w):=\int_{\mathbb{R}^{2}}((u,\nabla )v,w)dx
\end{equation}
is continuous on $\mathcal{H}^{1}\times\mathcal{H}^{1}\times\mathcal{H}^{1}$:
\begin{equation}
b(u,v,w)\leq C\Vert u\Vert _{1}\Vert v\Vert _{1}\Vert w\Vert _{1},
\label{s1r8}
\end{equation}
and for a fixed $u,v\in \mathcal{H}^{1}$  defines a bounded linear functional
$B(u,v)$ on $\mathcal{H}^{1}$:
$$
B(u,v):\mathcal{H}^{1}\to\mathcal{H}^{-1},\quad b(u,v,w)=\left\langle B(u,v),w\right\rangle
$$
with
\begin{equation}\label{s1r10}
\| B(u,v)\|_{-1}\leq C\|u\|_{1}\| v\|_{1}.
\end{equation}

In fact, we have more, namely, for any $\varepsilon>0$
\begin{equation}\label{s1r10eps}
\| B(u,v)\|_{-\varepsilon}\leq C_\varepsilon\|u\|_{1}\| v\|_{1},
\end{equation}
since by the H\"older inequality $u\nabla v\in L^{2-\varepsilon}$,
and by the Sobolev imbedding and duality $L^{2-\varepsilon}\supset
H^{-\varepsilon/(2-\varepsilon)}$.

We now define a weak solution of the system (\ref{s1r1}) in the space $%
\mathcal{H}^{1}.$  Let $g\in \mathcal{H}^{-1}.$

\begin{definition}
\label{s1def1}
{\rm A function $u(t,x )\in L^{\infty }(0,T;\mathcal{H}^{1})$ is
 a weak solution of (\ref{s1r1}) if for every $w\in \mathcal{H}^{1}$ the following equality
 holds in the sense of distributions $\mathcal{D}^{\prime }((0,T))$:
\begin{equation}
\frac{d}{dt}(u,w)+b(u,u,w)+r(u,w)=(g,w),  \label{s1r11}
\end{equation}%
or, equivalently, $u$ satisfies  the equation
\begin{equation}
\partial _{t}u+B(u,u)+ru=g  \label{s1r12}
\end{equation}%
in the space $\mathcal{D}^{\prime }((0,T);\mathcal{H}^{-1})$ of
distributions with values in $\mathcal{H}^{-1}.$ Both equalities
(\ref{s1r11}),  and (\ref{s1r12}) hold if and only if for every
smooth test function
$\varphi (t,x)\in C_{0}^{\infty}((0,T)\times\mathbb{R}^{2})$, $\func{div} \varphi=0$, the
following  integral identity holds
\begin{equation}
-\int_{0}^{T}(u,\partial _{t}\varphi )dt-
\int_{0}^{T}
\sum_{i=1}^{2}(u^{i}u,\partial _{i}\varphi)
dt
+r\int_{0}^{T}(u,\varphi )dt=\int_{0}^{T}(g,\varphi )dt.  \label{s1r13}
\end{equation}
}
\end{definition}

Let $u(t,x)\in L^{\infty }(0,T;\mathcal{H}^{1})$ be an arbitrary weak
solution of \eqref{s1r1}. Equation \eqref{s1r12} and inequality \eqref{s1r10} imply that
\begin{equation}\label{s1r13a}
\Vert \partial _{t}u\Vert _{-1}\leq \Vert B(u,u)\Vert _{-1}+r\Vert u\Vert
_{-1}+\Vert g\Vert _{-1}\leq C\Vert u\Vert _{1}^{2}+\Vert g\Vert _{-1}.
\end{equation}%
Therefore, $\partial _{t}u\in L^{\infty }(0,T;\mathcal{H}^{-1})$ and
$u\in C([0,T];\mathcal{H}^{-1})$.  Applying \cite[Lemma III.1.4]{TemNS}
we conclude that $u\in C_{w}([0,T];\mathcal{H}^{1})$.
Hence, the  initial condition is meaningful:
\begin{equation}
u|_{t=0}=u_{0}\in \mathcal{H}^{1}.  \label{s1r14}
\end{equation}

\begin{theorem} \label{s1the1}
Let $g\in \mathcal{H}^{1}$ and $u_{0}\in \mathcal{H}^{1}.$
Then the problem \eqref{s1r1}, \eqref{s1r14} has a weak solution
$u(t,x)\in L^{\infty }(0,T;\mathcal{H}^{1})$, which satisfies the following
estimate:
\begin{equation}
\Vert u(t)\Vert _{1}^{2}=\|u(t)\|^2+\|\nabla u(t)\|^2\leq \Vert u(0)\Vert _{1}^{2}e^{-rt}+r^{-2}\Vert
g\Vert _{1}^{2}.  \label{s1r15}
\end{equation}
\end{theorem}

\begin{proof}
The proof is standard and uses the vanishing viscosity
method (alternatively, one can use the Galerkin method with respect to
expanding family of bounded domains) (see \cite{BF, CVrj, CVZ}, see also \cite{Lio, Lad, Tem, TemNS, CoFo}).
Therefore we restrict ourselves to discussing a priori estimates for
weak solutions of the problem (\ref{s1r1}) and (\ref{s1r14}).

The $L^2$-norm estimate contained in  \eqref{s1r15} is obtained by taking the scalar product
of \eqref{s1r1} with $u$ and using the orthogonality relation
\begin{equation}\label{L2orth}
((u,\nabla v),v)=-\frac{1}{2}\int_{\mathbb{R}^2}|v|^2\func{div}u\,dx=0,
\end{equation}
which gives
\begin{equation}
\Vert u(t)\Vert ^{2}\leq \Vert u(0)\Vert ^{2}e^{-rt}+r^{-2}\Vert g\Vert ^{2}.
\label{s1r17}
\end{equation}
Next, we  estimate the $H^1$-norm and
 use the orthogonality relation $((u,\nabla u),\Delta u)=0$ 
specific to the two-dimensional case \cite{Tem}. In fact,
using the invariant expressions for the Laplacian~\eqref{invLapl} and
for the advection term 
\begin{equation}\label{s2r3}
(u,\nabla )u= \func{curl}u\times u+\frac{1}{2}\nabla |u|^{2},
\end{equation}
and setting
\begin{equation}\label{formulas}
\omega:=\func{curl}u=\partial_1u^2- \partial_2u^1,
\ u^\perp:=(-u^2,u^1),
\ \nabla^\perp:=(-\partial_2,\partial_1),
\end{equation}
we find that
\begin{equation}\label{H1orth}
\aligned
((u,\nabla u),\Delta u)=(\func{curl}u\times u,\func{curl}\func{curl}u)=
(\omega u^\perp,\nabla^\perp\omega)=\\=
(\omega u,\nabla\omega)=\frac12\int_{\mathbb{R}^2}u\nabla \omega^2dx=0.
\endaligned
\end{equation}
We take the scalar product of (\ref{s1r1}) with $\Delta u$ and
integrating by parts and using~\eqref{H1orth} we obtain
\begin{equation}
\Vert \nabla u(t)\Vert ^{2}\leq \Vert \nabla u(0)\Vert
^{2}e^{-rt}+r^{-2}\Vert \nabla g\Vert ^{2}.  \label{s1r20}
\end{equation}%
Summing \eqref{s1r17} and \eqref{s1r20} we finally have \eqref{s1r15}.

Finally, we observe that
since  the orthogonality relations \eqref{L2orth} and \eqref{H1orth}  are `bilinear', the formal
 argument above can be justified by  Galerkin
approximation solutions. Alternatively, we can consider more regular solutions of the corresponding
Navier--Stokes perturbation of this  system.
\end{proof}

\begin{corollary}\label{s1cor1}
Let $u(t,x)\in L^{\infty }(0,T;\mathcal{H}^{1})$ be an
arbitrary weak solution of \eqref{s1r1}. Then the function
$\Vert u(t)\Vert^{2}$ is absolutely continuous and the following energy
equality holds%
\begin{equation}\label{s1r21}
\frac{1}{2}\frac{d}{dt}\Vert u(t)\Vert ^{2}+r\Vert u(t)\Vert ^{2}=(g,u(t)).
\end{equation}
\end{corollary}
\begin{proof}
The established regularity
$u(t,x)\in L^{\infty }(0,T;\mathcal{H}^{1})\subset L^{2 }(0,T;\mathcal{H}^{1})$
of the weak solution and
\eqref{s1r13a} giving
$\partial_t u(t,x)\in L^{\infty }(0,T;\mathcal{H}^{-1})
\subset L^{2 }(0,T;\mathcal{H}^{-1})$
show that we can take the scalar product of
\eqref{s1r1} with $u$ to obtain
\begin{equation}
(\partial _{t}u,u)+r(u,u)=(g,u).  \label{s1r15a}
\end{equation}%
It now follows from \cite[Lemma~III.1.2]{TemNS} that%
\begin{equation*}
(\partial _{t}u,u)=\frac{1}{2}\frac{d}{dt}\Vert u\Vert ^{2}\ \text{for a.e. }%
t\in \lbrack 0,T]
\end{equation*}%
and the function $\Vert u(t)\Vert ^{2}$ is absolutely continuous. Therefore we
 obtain from
(\ref{s1r15a})  the  \emph{energy} equality \eqref{s1r21}.
\end{proof}

\begin{remark}\label{R:1}
{\rm
If we assume additional regularity of the solution $u$,
namely, $u(t,x)\in L^{2}(0,T;\mathcal{H}^{2})$,
then arguing as in Corollary~\ref{s1cor1} one can show that
the following \emph{enstrophy} equality holds
\begin{equation}\label{s1r19}
\frac{1}{2}\frac{d}{dt}\|\nabla u(t)\|^{2}+r\|\nabla u(t)\|^{2}=(\nabla g,\nabla u(t)).
\end{equation}
Summing (\ref{s1r21}) and (\ref{s1r19}) and taking into account that
$\|\nabla u\|^2=\|\func{curl}u\|^2$,  we obtain the equality
for the full $H^1$-norm
\begin{equation}\label{s1r22}
\aligned
\frac{1}{2}\frac{d}{dt}\bigl(\| u(t)\|^{2}+&\|\func{curl}u\|^2\bigr)+
r\left(\| u(t)\|^{2}+\|\func{curl}u\|^2\right)=\\=
&(u(t),g)+(\func{curl}u(t),\func{curl}g)\equiv (u(t),g)_{1}.
\endaligned
\end{equation}

In the next section we establish the enstrophy equality \eqref{s1r19}
and, hence, equality~\eqref{s1r22}   for an \emph{arbitrary} weak
solution $u\in L^{\infty}(0,T;\mathcal{H}^{1})$
of the problem (\ref{s1r1}), (\ref{s1r14})
(not necessarily constructed by, say, the Galerkin method). This will be the key
tool  in the construction of the strong trajectory
attractor for the damped-driven 2D Euler system.
}
\end{remark}

\begin{remark}\label{s1rem2}
{\rm
The uniqueness of a weak solution of the problem (\ref{s1r1})
and (\ref{s1r14}) for $g\in \mathcal{H}^{1}$ and $u_{0}\in \mathcal{H}^{1}$
in the space $L^{\infty }(0,T;\mathcal{H}^{1})$ is not proved. The situation
is the same as in the case of the  classical conservative 2D Euler system with $r=0$.
Therefore our goal is  to construct strong global and trajectory attractors for
the system~(\ref{s1r1}) with possibly non-unique solutions.
}
\end{remark}

\section{Vorticity equation and enstrophy equality \label{Sec2}}

Let $u(t,x)\in L^{\infty }(0,T;\mathcal{H}^{1})$ be an arbitrary
weak solution of  system (\ref{s1r1}). We consider the following
(scalar) function called the \emph{vorticity}
\begin{equation*}
\omega (t,x):=\func{curl}u(t,x)=\partial _{1}u^{2}(t,x)-\partial
_{2}u^{1}(t,x).
\end{equation*}
We clearly have $\omega \in L^{\infty }(0,T;H)$. We now derive the
equation for $\omega (t,x)$.
Using \eqref{s2r3},
where $\func{curl}u\times u=\omega u^\perp$, see~\eqref{formulas},
we rewrite the system \eqref{s1r1} in the equivalent form
\begin{equation}
\left\{
\begin{array}{l}
\partial _{t}u^{1}-\omega u^{2}+\partial _{1}p+\dfrac{1}{2}\partial
_{1}|u|^{2}+ru^{1}=g^{1}, \\
\partial _{t}u^{2}+\omega u^{1}+\partial _{2}p+\dfrac{1}{2}\partial
_{2}|u|^{2}+ru^{2}=g^{2}, \\
\func{div}u=0.
\end{array}
\right.   \label{s2r4}
\end{equation}
Taking the (distributional) derivative $\partial _{2}$ of the first equation
of the system \eqref{s2r4} and $\partial _{1}$ of the second equation, we
obtain
\begin{equation*}
\left\{
\begin{array}{l}
\partial _{t}\partial _{2}u^{1}-\partial _{2}(\omega u^{2})+\partial
_{2}\partial _{1}p+\dfrac{1}{2}\partial _{2}\partial
_{1}|u|^{2}+r\partial _{2}u^{1}=\partial _{2}g^{1}, \\
\partial _{t}\partial _{1}u^{2}+\partial _{1}(\omega u^{1})+\partial
_{1}\partial _{2}p+\dfrac{1}{2}\partial _{1}\partial
_{2}|u|^{2}+r\partial _{1}u^{2}=\partial _{1}g^{2}.%
\end{array}%
\right.
\end{equation*}%
We now subtract the first equation from the second  and obtain%
\begin{equation*}
\partial _{t}\left( \partial _{1}u^{2}-\partial _{2}u^{1}\right) +\partial
_{1}(\omega u^{1})+\partial _{2}(\omega u^{2})+r\left( \partial
_{1}u^{2}-\partial _{2}u^{1}\right) =\partial _{1}g^{2}-\partial
_{2}g^{1}.
\end{equation*}%
Here we have used the commutation property $\partial _{2}\partial
_{1}=\partial _{1}\partial _{2}$ for the distributional
derivatives. The resulting equation for
the vorticity $\omega $ reads%
\begin{equation}\label{s2r6}
\partial _{t}\omega +\sum_{j=1}^{2}\partial _{j}(\omega u^{j})+r\omega =%
\func{curl}g(x).
\end{equation}%
This equation holds in the  sense of distributions, that is, for every smooth test
function $\varphi(t,x) \in C_{0}^{\infty }((0,T)\times\mathbb{R}^{2})$
 the following integral identity holds:
\begin{equation}\label{s2r7}
\aligned
-\int_{0}^{T}\omega \partial _{t}\varphi dt-&\int_{0}^{T}
\int_{\mathbb{R}^{2}}\omega \sum_{j=1}^{2}u^{j}\partial _{j}\varphi dxdt+\\
&+r\int_{0}^{T}\int_{\mathbb{R}^{2}}\omega \varphi dxdt=
\int_{0}^{T}\int_{\mathbb{R}^{2}}\varphi\func{curl}g dxdt.
\endaligned
\end{equation}
It is well known that if $f_{1},f_{2}\in H^{1},$ then $\partial
_{i}(f_{1}f_{2})=(\partial _{i}f_{1})f_{2}+f_{1}(\partial _{i}f_{2})\in
L_{loc}^{1}$ is the distributional derivative  $\partial _{i}$ of the
product $f_{1}f_{2}$. Recall that $u\in L^{\infty }(0,T;\mathcal{H}^{1})$ and 
$\func{div}u=0$. Therefore we have
\begin{equation*}
\sum_{j=1}^{2}\partial _{j}\left( u^{j}\varphi \right) = \varphi\func{div} u +\sum_{j=1}^{2}u^{j}\partial
_{j}\varphi =\sum_{j=1}^{2}u^{j}\partial _{j}\varphi .
\end{equation*}
Consequently, we can rewrite identity \eqref{s2r7} as follows
\begin{equation}\label{weakomega}
\aligned
-\int_{0}^{T}\omega \partial _{t}\varphi dt-&\int_{0}^{T}
\int_{\mathbb{R}^{2}}\omega \sum_{j=1}^{2}\partial _{j}\left( u^{j}\varphi \right)
dxdt+\\+&r\int_{0}^{T}\int_{\mathbb{R}^{2}}\omega \varphi dxdt=
\int_{0}^{T}\int_{\mathbb{R}^{2}}\varphi\func{curl}g\varphi dxdt.
\endaligned
\end{equation}%

We have proved the following result.

\begin{proposition}
\label{s2pro1} Let $g\in \mathcal{H}^{1}$. If
$u(t,x)\in L^{\infty}(0,T;\mathcal{H}^{1})$ is a weak solution of \eqref{s1r1}, then the
corresponding vorticity function $\omega =\func{curl}u$ is the
solution of the equation
\begin{equation}
\partial _{t}\omega +(u,\nabla )\omega +r\omega =\func{curl}g(x).
\label{s2r8}
\end{equation}%
 in the sense of the integral identity~\eqref{weakomega}.
\end{proposition}

We note that in equation \eqref{s2r6} we have 
$u\in L^{\infty}(0,T;\mathcal{H}^{1})$ so that 
$\omega \in L^{\infty}(0,T;L^{2})$. Using the embedding 
$H^{1}(\mathbb{R}^2)\subset L^{q}(\mathbb{R}^2)$ for any $q\geq 2$,
 we observe that the product
$\omega \cdot u^{j}$ belongs to 
$L^{\infty }(0,T;L^{2-\varepsilon})$ for any $\varepsilon >0$. Therefore, 
$\sum_{j=1}^{2}\partial _{j}(\omega u^{j})\in L^{\infty }(0,T;W^{-1,2-\varepsilon })$ and
the equation (\ref{s2r6}) holds the space of distributions\
$\mathcal{D}^{\prime }((0,T);W^{-1,2-\varepsilon })$. We also note
that  since $\func{div}u=0$, the weak  formulations of \eqref{s2r6}
and \eqref{s2r8}, namely, \eqref{s2r7} and \eqref{weakomega},
respectively, are equivalent.

We now formulate the main theorem of this section.

\begin{theorem} \label{s2the1}
Let  $\omega (t,x)\in L^{\infty }(0,T;L^{2})$ be
a  solution of the linear transport equation \eqref{s2r8} in the sense of
\eqref{s2r7} (or \eqref{weakomega}), where
$u(t,x)\in L^{\infty }(0,T;\mathcal{H}^{1})$ is an arbitrary function
(not necessarily a weak solution of \eqref{s1r1}). Then the function $\|\omega (t)\|^{2}$
is absolutely continuous and satisfies the differential equation
\begin{equation}\label{s2r9}
\frac{1}{2}\frac{d}{dt}\Vert \omega (t)\Vert ^{2}+r\Vert \omega (t)\Vert
^{2}=(\omega (t),\func{curl}g)\ \ \text{\rm a.\,e. on}\ (0,T).
\end{equation}
\end{theorem}

\begin{proof}
This follows from the proof of \cite[Theorem II.2 and equation (26)]{DiP-Lio},
the only difference being that in \cite{DiP-Lio} the unforced case
($g=0$) is considered.
\end{proof}

\begin{corollary}\label{s2cor1}
Let $g\in \mathcal{H}^{1}$.  If $u(t,x)\in L^{\infty }(0,T;\mathcal{H}^{1})$
is a weak solution of \eqref{s1r1}, then
the function $\|u(t)\|^2_1=\|u(t)\|^{2}+\|\nabla u(t)\|^2$ is absolutely continuous
and
\begin{equation} \label{s2r10}
\aligned
\frac{1}{2}\frac{d}{dt}\bigl(\|u(t)\|^{2}+\|\nabla u(t)\|^2\bigr)
+r\left(\|u(t)\|^{2}+\|\nabla u(t)\|^2\right)=\\=(u(t),g)+(\nabla u(t),\nabla g)
\endaligned
\end{equation}
almost everywhere on  $t\in(0,T)$.  In addition,
$u\in C([0,T];\mathcal{H}^{1})$ and
$\partial _{t}u\in C([0,M];\mathcal{H}^{-\varepsilon})$ for every $\varepsilon>0$.
\end{corollary}
\begin{proof}
We recall that $\|\nabla u\| =\|\func{curl}u\|$ for all $u\in
\mathcal{H}^{1}$ (see \eqref{s1r3}). Therefore, the equalities \eqref{s1r21}
and \eqref{s2r9} prove  the full $H^1$-norm equality \eqref{s1r22}
(equivalently, \eqref{s2r10})  for any weak
solution of the damped Euler system \eqref{s1r1}.

Next,  we set $F(u)=-B(u,u)-ru+g$. Then $\partial _{t}u=F(u)$, see
\eqref{s1r12}. It follows from \eqref{s1r10eps} that the operator
$F$ is continuous from $\mathcal{H}^{1}$ to
$\mathcal{H}^{-\varepsilon}$. Therefore $\partial_tu$ is  bounded
in $L^\infty(0,T;\mathcal{H}^{-\varepsilon})$, and, hence, $u\in
C([0,T];\mathcal{H}^{-\varepsilon})$. Since $u$ is bounded in
$L^\infty(0,T;\mathcal{H}^{1})$, it follows that $u$ is weakly
continuous in $\mathcal{H}^{1}$, see, for instance,
\cite[Lemma~III.1.4]{TemNS}. Now the continuity with respect to
time of the norms $\|u(t)\|$ and $\|\nabla u(t)\|$ proved in
Corollary~\ref{s1cor1} and Theorem~\ref{s2the1} and the  weak
continuity of $u$ in $\mathcal{H}^{1}$ imply the strong continuity:
$u\in C([0,T];\mathcal{H}^{1})$. This also gives that $\partial
_{t}u\in C([0,T];\mathcal{H}^{-\varepsilon})$.
\end{proof}

\section{Weak trajectory attractor for damped Euler system \label{Sec3}}

We consider the following two linear spaces
\begin{equation*}
\mathcal{F}_{+}^{\infty }=L^{\infty }(\mathbb{R}_{+};\mathcal{H}^{1})\quad
\text{and}\quad \mathcal{F}_{+}^{\mathrm{loc}}=L_{\mathrm{loc}}^{\infty }(%
\mathbb{R}_{+};\mathcal{H}^{1}).
\end{equation*}%
Recall that
$z(t,x)\in L_{\mathrm{loc}}^{\infty }(\mathbb{R}_{+};\mathcal{H}^{1})$ if and only if
 $z(t,x)\in L^{\infty }(0,M;\mathcal{H}^{1})$ for all $M>0$.
We clearly have
$\mathcal{F}_{+}^{\infty }\subset \mathcal{F}_{+}^{\mathrm{loc}}$.

The space $\mathcal{F}_{+}^{\mathrm{loc}}$ is equipped with the following
local weak topology $\Theta _{+}^{\mathrm{loc,w}}$. By definition, a
sequence $z_{n}\in\mathcal{F}_{+}^{\mathrm{loc}}$ converges in $\Theta _{+}^{\mathrm{loc,w}}$
to an element $z$   if  $z_{n}\rightharpoondown z$ $\ast $-weakly in
$L^{\infty }(0,M;\mathcal{H}^{1})$ for every $M>0$.
Sometimes, when it causes no ambiguity,  we  denote by $\Theta _{+}^{\mathrm{loc,w}}$
the space $\mathcal{F}_{+}^{\mathrm{loc}}$ with topology $\Theta _{+}^{\mathrm{loc,w}}$.

We now define the \emph{trajectory space} $\mathcal{K}_{+}$ for the equation
(\ref{s1r1}).

\begin{definition}\label{s3def1}
{\rm The trajectory space $\mathcal{K}_{+}$ consists of all functions from $%
\mathcal{F}_{+}^{\infty }=L^{\infty }(\mathbb{R}_{+};\mathcal{H}^{1})$ that
are weak solutions of  system (\ref{s1r1}) on every interval $(0,M)$.
The elements from $\mathcal{K}_{+}$ are called the \emph{trajectories}.
}
\end{definition}

It follows from Theorem \ref{s1the1} that for any $u_{0}\in\mathcal{H}^{1}$ there is at least
one trajectory $u\in \mathcal{K}_{+}$ such that $u(0)=u_{0}$.
Thus, $\mathcal{K}_{+}$ is non-empty.

\begin{proposition} \label{s3pro1}
 The trajectory space $\mathcal{K}_{+}$ is sequentially closed in the
topology $\Theta _{+}^{\mathrm{loc,w}}$ and, in addition,
 $\mathcal{K}_{+}\subset C_{b}(\mathbb{R}_{+};\mathcal{H}^{1})$.
\end{proposition}

\begin{proof}
Let $\{u_{n}(t,x)\}$ be a sequence from $\mathcal{K}_{+}$ and let
$u_{n}(t,x)\rightarrow u(t,x)$ in $\Theta _{+}^{\mathrm{loc,w}}$
for some $u\in $\textrm{$\mathcal{F}_{+}^{\infty }$}$,$ i.e.,
\begin{equation}
u_{n}(t,x)\rightharpoondown u(t,x)\quad \text{$\ast$-weakly in }L^{\infty
}(0,M;\mathcal{H}^{1}),\ \forall M>0.  \label{S3r1}
\end{equation}
We claim that $u\in \mathcal{K}_{+}$. The functions $u_{n}(t,x)$
satisfy (\ref{s1r1}) in the sense of distributions, that is,
\begin{equation}
-\int_{0}^{M}(u_{n},\partial _{t}\varphi )dt-\int_{0}^{M}
\sum_{i=1}^{2}(u_{n}^{i}u_{n},\partial _{i}\varphi )dt
+r\!\int_{0}^{M}(u_{n},\varphi )dt=\int_{0}^{M}\!(g,\varphi )dt
\label{S3r2}
\end{equation}
for every $\varphi (t,x)\in C_{0}^{\infty }((0,M)\times
\mathbb{R}^{2})$, $\func{div}\varphi =0.$ Let us show that $u(t,x)$
is also a distributional solution of (\ref{s1r1}). We fix an arbitrary
$M>0$ and $\varphi \in C_{0}^{\infty }((0,M)\times
\mathbb{R}^{2})$, $\func{div}\varphi =0.$ Let the ball
$B_{0}^{R}:=\{x\in \mathbb{R}^{2},|x|\leq R\}$ contain the support
of $\varphi$.

It follows from (\ref{S3r1}) that $\{u_{n}(t,x)\}$ is bounded in
$L^{\infty }(0,M;\mathcal{H}^{1})$. Then due to \eqref{s1r13a}
$\{\partial _{t}u_{n}(t,x)\}$ is bounded in $L^{\infty}(0,M;\mathcal{H}^{-1})$.
 Passing to a subsequence, we may assume that
\begin{equation*}
\partial _{t}u_{n}(t,x)\rightharpoondown \partial _{t}u(t,x)\quad
\text{$\ast$-weakly in }L^{\infty }(0,M;\mathcal{H}^{-1}).
\end{equation*}
Applying Aubin compactness theorem (see, for instance, \cite{BoFa},
\cite{CVbook}) and restricting the functions to the balls
$B_{0}^{R}$, we conclude that
\begin{equation}
u_{n}\rightarrow u\quad \text{strongly in }L^{2}(0,M;\mathcal{H}(B_{0}^{R})).
\label{S3r4}
\end{equation}
Recall that $L^{2}(0,M;\mathcal{H}(B_{0}^{R}))\subset
L^{2}((0,M)\times B_{0}^{R})^{2}$ and therefore
\begin{equation}
u_{n}(t,x)\rightarrow u(t,x)\quad \text{for a.e. }(t,x)\in (0,M)\times
B_{0}^{R}.  \label{S3r5}
\end{equation}
Consider the trilinear term in the left-hand side of (\ref{S3r2}).
It follows from (\ref{S3r5}) that for $i=1,2$
\begin{equation}
u_{n}^{i}(t,x)u_{n}(t,x)\rightarrow u^{i}(t,x)u(t,x)\quad \text{for a.e. }
(t,x)\in (0,M)\times B_{0}^{R}.  \label{S3r6}
\end{equation}
Recall that the sequence $\{u_{n}\}$ is bounded in $L^{\infty }(0,M;\mathcal{H}^{1})$.
 Then, due to the embedding $\mathcal{H}^{1}\subset L^{4}(B_{0}^{R})^{2}$, we see that
\begin{equation}
\{u_{n}^{i}u_{n}\}\quad \text{is bounded in }L^{\infty
}(0,M;L^{2}(B_{0}^{R})^{2})  \label{S3r7}
\end{equation}
and in $L^{2}((0,M)\times B_{0}^{R})^{2}$ as well. Applying
 \cite[Lemma ~1.1.3]{Lio} on weak convergence, we conclude
from \eqref{S3r6} and \eqref{S3r7} that
$$
u_{n}^{i}u_{n}\rightharpoondown u^{i}u\quad \text{weakly in  }
L^{2}((0,M)\times B_{0}^{R})^{2}.
$$
Therefore, since $\mathrm{supp\ }\varphi\subset B_{0}^{R},$
$$
\int_{0}^{M}\sum_{i=1}^{2}(u_{n}^{i}u_{n},\partial_{i}\varphi )dt\rightarrow \int_{0}^{M}
\sum_{i=1}^{2}(u^{i}u,\partial _{i}\varphi )dt\quad \text{as }n\rightarrow
\infty.
$$
Finally, we see that (\ref{S3r4}) allows to pass to the limit as
$n\rightarrow \infty $ in the remaining  terms of (\ref{S3r2}) and we
obtain the equality
$$
-\int_{0}^{M}(u,\partial _{t}\varphi )dt-\int_{0}^{M}
\sum_{i=1}^{2}(u^{i}u,\partial _{i}\varphi)dt+r\int_{0}^{M}(u,\varphi )dt=
\int_{0}^{M}(g,\varphi )dt.
$$
Since the function $\varphi $ and the number $M>0$ are arbitrary,
the function $u$ is a weak solution of (\ref{s1r1}), that is, $u\in
\mathcal{K}_{+}.$ Hence, $\mathcal{K}_{+}$ is sequentially closed in the
topology $\Theta _{+}^{\mathrm{loc,w}}$.

The embedding $\mathcal{K}_{+}\subset
C_{b}(\mathbb{R}_{+};\mathcal{H}^{1})$ follows from Corollary
\ref{s2cor1}.
\end{proof}

\begin{remark}\label{s3rem1}
 Any ball
$\mathcal{B}(0,R)=
\left\{ z\in \mathcal{F}_{+}^{\infty },\ \| z\| _{\mathcal{F}_{+}^{\infty }}\le R\right\} $ is a
compact subset in the topology $\Theta _{+}^{\mathrm{loc,w}}$ and the
corresponding space is metrizable and complete (see \cite{CVbook, RobRob}).
Therefore the intersection of the trajectory space $\mathcal{K}_{+}$
with $\mathcal{B}(0,R)$ is compact in the topology $\Theta _{+}^{\mathrm{loc,w}}$.

On the other hand, the entire space $\mathcal{F}_{+}^{\mathrm{loc}}$ and
its subspace $\mathcal{F}_{+}^{\infty }$ are not metrizable in
the topology $\Theta _{+}^{\mathrm{loc,w}}$.
\end{remark}

We now consider the translation semigroup $\{T(h)\}:=\{T(h),h\geq 0\}$
acting on the spaces $\mathcal{F}_{+}^{\mathrm{loc}}$ and
$\mathcal{F}_{+}^{\infty }$ by the formula $T(h)z(t)=z(t+h)$. The translations $T(h)$
clearly map the trajectory space $\mathcal{K}_{+}$ into itself:
$$
T(h)\mathcal{K}_{+}\subseteq \mathcal{K}_{+},\ \forall h\geq 0.
$$
Indeed, if $u(t),t\geq 0,$ is a weak solution of \eqref{s1r1} on every
interval $(0,M)$ then the function $u(t+h),t\geq 0$ is also a weak solution
of \eqref{s1r1} since the equation is autonomous.

\begin{proposition}
\label{s3pro2} The ball $\mathcal{B(}0,R_{0})$ in $\mathcal{F}_{+}^{\infty }$
with radius $R_{0}=\sqrt{2}r^{-1}\Vert g\Vert _{1}$ is an absorbing set of
the semigroup $\{T(h)\}$ acting on $\mathcal{K}_{+},$ that is, for any
 subset $B\subset \mathcal{K}_{+}$ bounded in $\mathcal{F}_{+}^{\infty }$
there is a number $h_{1}=h_{1}(B)$ such that $T(h)B\subset \mathcal{B(}0,R_{0})$ for
all $h\geq h_{1}$.
\end{proposition}

\begin{proof}
It follows from Corollary~\ref{s2cor1} that any weak solution
satisfies the enstrophy equality~\eqref{s2r10}. This shows that
any weak solution satisfies  the dissipative estimate~\eqref{s1r15}.
(In Theorem~\ref{s1the1} we proved that a weak solution
constructed by the Galerkin approximations satisfies~\eqref{s1r15}.)
It now follows from inequality \eqref{s1r15} that if $B$ is a bounded set of trajectories
from $\mathcal{K}_{+}$ and $\| u\| _{\mathcal{F}_{+}^{\infty }}\leq R$
for all $u\in B$, then
\begin{equation*}
\Vert T(h)u\Vert _{\mathcal{F}_{+}^{\infty }}^{2}=\underset{t\geq 0}{\mathrm{%
ess}\sup }\Vert u(h+t)\Vert _{1}^{2}\leq R^{2}e^{-rh}+r^{-2}\Vert g\Vert
_{1}^{2}\leq 2r^{-2}\Vert g\Vert _{1}^{2}=R_{0}^{2},
\end{equation*}%
for all $h\ge h_1$, where $e^{-rh_{1}}\leq r^{-2}\Vert g\Vert _{1}^{2}/R^{2}.$ That is, $%
T(h)u\subset \mathcal{B(}0,R_{0})$ for $h\geq h_{1}.$ Hence, the ball
 $\mathcal{B}(0,R_{0})$ is absorbing.
\end{proof}

We now recall the definition of the weak trajectory attractor (see \cite{CVbook, CV2, VC-UMN}).

\begin{definition} \label{s3def2}
{\rm
A set $\mathfrak{A}\subset \mathcal{K}_{+}$ is called a
weak trajectory attractor of (\ref{s1r1}) if

\begin{enumerate}
\item $\mathfrak{A}$ is compact in $\Theta _{+}^{\mathrm{loc,w}}$ and
bounded in $\mathcal{F}_{+}^{\infty }$;

\item $\mathfrak{A}$ is strictly invariant, that is, $T(h)\mathfrak{A}=%
\mathfrak{A}$ for all $h\geq 0$;

\item $\mathfrak{A}$ attracts any bounded trajectory set $B\subset \mathcal{K%
}_{+}$, that is,  for any neighborhood $\mathcal{O}(\mathfrak{A})$ of
$\mathfrak{A}$ in the topology $\Theta _{+}^{\mathrm{loc,w}},$ there is
$h_{1}=h_{1}(B,\mathcal{O})\geq 0$ such that
$T(h)B\subset \mathcal{O}(\mathfrak{A})$ for all $h\geq h_{1}$.
\end{enumerate}
}
\end{definition}

Similarly  to the space $\mathcal{F}_{+}^{\mathrm{loc}}$ and the space
$\mathcal{F} _{+}^{\infty }$ with topology $\Theta _{+}^{\mathrm{loc,w}}$, we consider the
space $\mathcal{F}^{\infty }=L^{\infty }(\mathbb{R};\mathcal{H}^{1})$ and the space
$\mathcal{F}^{\mathrm{loc}}=L_{\mathrm{loc}}^{\infty }(\mathbb{R};\mathcal{H}^{1})$ with
topology $\Theta ^{\mathrm{loc,w}}$ defined on the entire time-axis.

\begin{definition} \label{s3def3}
{\rm
The kernel $\mathcal{K}$ of the system (\ref{s1r1}) is the
union of all functions $u(t,x)\in L^{\infty }(\mathbb{R};\mathcal{H}^{1})$
that are weak solutions of (\ref{s1r1}) on any interval $(-M,M)$. The elements
of $\mathcal{K}$ are often called {complete bounded trajectories} of
the system (\ref{s1r1}).
}
\end{definition}

\begin{proposition}
\label{s3pro3} The kernel $\mathcal{K}$ of  system (\ref{s1r1}) belongs
to $C_{b}(\mathbb{R};\mathcal{H}^{1}).$
\end{proposition}

The next theorem proves the existence and describes the structure of the
weak trajectory attractor for damped 2D Euler system.

\begin{theorem}
\label{s3the1}Let the external force $g\in \mathcal{H}^{1}$. Then the
equation \eqref{s1r1} has the weak trajectory attractor and
\begin{equation} \label{s3r1}
\mathfrak{A}=\Pi _{+}\mathcal{K},
\end{equation}
where $\mathcal{K}$ is the kernel of \eqref{s1r1} and $\Pi _{+}$
denotes the restriction operator onto the semiaxis $\mathbb{R}_{+}.$
Besides, $\mathfrak{A}$ is bounded in $C_{b}(\mathbb{R}_{+};\mathcal{H}^{1})$
and the following inequality holds:
\begin{equation}
\Vert u\Vert _{C_{b}(\mathbb{R}_{+};\mathcal{H}^{1})}\leq r^{-1}\Vert g\Vert
_{1},\ \forall u\in \mathfrak{A}.  \label{s3r2}
\end{equation}
\end{theorem}

\begin{proof}
To apply the general theorem on the existence and the structure of
trajectory attractors (see, e.g., \cite{CVbook}), we have to verify
that our trajectory space $\mathcal{K}_{+}$ contains, for the
translation semigroup $\{T(h)\}$, an attracting (or absorbing) set
which is bounded in $\mathcal{F}_{+}^{\infty }$ and compact in the
topology $\Theta _{+}^{\mathrm{loc,w}}$. We have proved in
Propositions \ref{s3pro1} and \ref{s3pro2} that such absorbing set
is $\mathcal{B(}0,R_{0})\cap \mathcal{K}_{+}$. See also Remark
\ref{s3rem1}. Therefore, there exists a trajectory attractor
$\mathfrak{A}$ for system \eqref{s1r1} with structure \eqref{s3r1}.
We also have that $\mathfrak{A\subset
}C_{b}(\mathbb{R}_{+};\mathcal{H}^{1}),$ while
 inequality \eqref{s3r2} follows from \eqref{s1r15}.
\end{proof}

Using  Corollary \ref{s2cor1}
 we obtain the following estimate for the time derivatives
of the solutions lying on the  attractor.

\begin{corollary}
\label{s3cor1} For any $u\in \mathfrak{A},$ the following inequality holds%
\begin{equation}
\Vert \partial _{t}u\Vert _{C_{b}(\mathbb{R}_{+};\mathcal{H}^{-\varepsilon})}\leq
C_{\varepsilon},  \label{s3r3}
\end{equation}%
where $\varepsilon>0$ and $C_{\varepsilon}$ depends on $r$, $\Vert g\Vert _{1}$ and is independent of $u$.
\end{corollary}

\begin{corollary} \label{s3cor2}
The trajectory attractor $\mathfrak{A}$ is compact in the
space $C_{\mathrm{loc}}(\mathbb{R}_{+};\mathcal{H}_{\mathrm{loc}}^{1-\delta })$ for any
$\delta>0$.  Moreover, for any bounded set $B\subset \mathcal{K}_{+}$%
\begin{equation}
\mathrm{dist}_{C([0,M];\mathcal{H}_{\mathrm{loc}}^{1-\delta })}\left( T(h)B,\mathfrak{A}%
\right) \rightarrow 0\ (h\rightarrow +\infty ),\ \forall M>0.  \label{s3r4}
\end{equation}%
(Here and below, $\mathrm{dist}_{\mathcal{M}}(X,Y)$ denotes the Hausdorff
semidistance from a set $X$ to a set $Y$ in a metric space $\mathcal{M}$).
\end{corollary}
\begin{proof}
In fact, if $B\subset \mathcal{K}_{+}$ is a bounded set
of trajectories $u(t,x)$, then any sequence
$\{u_{n}\}\subset B$ is bounded in $C([0,M];\mathcal{H}^{1})$ and, due to
Corollary~\ref{s3cor1}, the sequence $\{\partial _{t}u_{n}\}$ is bounded in
the space $C([0,M];\mathcal{H}^{-1})$. Applying
Aubin--Lions--Simon compactness theorem
(see, for instance,  \cite{BoFa}, \cite{CVbook})
and restricting  functions to the balls $B_0^R:=\{x\in\mathbb{R}^2,|x|\le R\}$, we conclude that $\{u_{n}\}$ is precompact in $C([0,M];%
\mathcal{H}^{1-\delta }(B_0^R))$ for any $\delta >0$ and $R<\infty$. Therefore, we have the compact
embedding
of every set $B\subset \mathcal{K}_{+}$ of bounded trajectories into
\begin{equation*}
C_{\mathrm{loc}}(\mathbb{R}_{+};\mathcal{%
H}^{1-\delta }_{\mathrm{loc}}),\ \text{for every}\  \delta>0.
\end{equation*}
This property implies \eqref{s3r4}. In particular, taking $B=\mathfrak{A}$
we obtain that $\mathfrak{A}$ is compact in
$C_{\mathrm{loc}}(\mathbb{R}_{+};\mathcal{H}^{1-\delta}_\mathrm{loc})$.
\end{proof}

In the next section we prove that Corollary~\ref{s3cor2} holds for the
limiting  $\delta =0$, furthermore, we  remove the locality condition
with respect to the spatial variable $x\in\mathbb{R}^2$.

\section{Strong global and trajectory attractors \label{Sec4}}

Our main result in this section relies heavily on the construction of
the strong \emph{global} attractor for the system \eqref{s1r1}, whose definition
we now recall. Since, generally speaking, the solutions of the system
are non-unique, we have to accordingly modify the classical definition
\cite{BV, CVbook, CoFo,Tem}.

\begin{definition} $($\cite{CVbook,VC-UMN}$)$\label{s3defglob}
{\rm
A set $\mathcal{A}\subset\mathcal{H}^1$ is called the global attractor of  system
\eqref{s1r1} if
\begin{enumerate}
  \item $\mathcal{A}$ is compact in $\mathcal{H}^1$:  $\mathcal{A}\Subset\mathcal{H}^1$;
  \item $\mathcal{A}$ is attracting: for every  set $B\subset\mathcal{K}^+$  that is bounded in
  $C_b(\mathbb{R}^+;\mathcal{H}^{1})$
$$
{\mathrm{dist}}_{\mathcal{H}^1}(B(t),\mathcal{A})\to0\ \text{as}\ t\to\infty;
$$
  \item $\mathcal{A}$ is the minimal set (with respect to inclusion)  among all compact
  attracting sets.
\end{enumerate}
Here $B(t)$  is the (well-defined) section at time $t\ge0$ of a set of bounded trajectories
$B$:
$$
B(t)=\{u(t),\ u\in B\}\subset\mathcal{H}^1.
$$
}
\end{definition}

We now formulate the main result of this work (in the $W^{1,2}$
case) in terms of the global and trajectory attractors. We first
consider the global attractor.

In Theorem \ref{s3the1} we
have shown  that the weak  trajectory attractor $\mathfrak{A}$ exists and is bounded in
$C_{b}(\mathbb{R}_{+};\mathcal{H}^{1})$.

We set
\begin{equation*}
\mathcal{A}:=\mathfrak{A}(0)=\{u(0),\ u\in \mathfrak{A}\}\subset \mathcal{H}%
^{1}.
\end{equation*}

\begin{theorem}\label{T:glob}
The set $\mathcal{A}$ is the global attractor.
\end{theorem}

\begin{proof}
We fix a bounded trajectory set $B\subset \mathcal{K}_{+}$. Let
$h_{n}\rightarrow +\infty$,  and for each $n$  let  $u_{n}\in T(h_n)B$ be an arbitrary
$h_n$-shifted trajectory.
By Theorem \ref{s3the1}, since $\mathfrak{A}$ is the weak trajectory
attractor, passing to a subsequence of $\{h_{n}\}$ which we label the same,
we may assume that

\begin{enumerate}
\item there is $u\in \mathfrak{A}$ such that $u_{n}\rightharpoondown u\ $%
in the (weak) topology $\Theta _{+}^{\mathrm{loc,w}};$

\item $u_{n}(0)\rightharpoondown u(0)$ weakly in $\mathcal{H}^{1};$

\item there are weak solutions $U_{n}(t)$ of equation (\ref{s1r1}) defined
for $[-h_{n},+\infty )$ such that $U_{n}(t)=u_{n}(t)$ for all $t\geq 0$ and
\begin{equation}
\Vert U_{n}(t)\Vert _{\mathcal{H}^{1}}\leq C,\ \forall t\geq -h_{n},
\label{s4r3}
\end{equation}%
where $C=C(B)$ is independent of $n;$

\item there is a bounded complete trajectory $U\in $ $\mathcal{K}$ such that
$U_{n}\rightharpoondown U\ $in the (weak) topology $\Theta ^{%
\mathrm{loc,w}}$ and $U(t)=u(t)$ for all $t\geq 0.$
\end{enumerate}

To prove the theorem we have to show that
\begin{equation}\label{strongH1}
u_{n}(0)\rightarrow u(0)\ \text{strongly in }\mathcal{H}^{1}\ \text{as }%
n\rightarrow \infty .
\end{equation}%
In fact, since $u_{n}(0)\rightharpoondown u(0)\ $weakly in $\mathcal{H}^{1},$
it is sufficient to prove that
\begin{equation}\label{s4r4}
\Vert u_{n}(0)\Vert _{1}\rightarrow \Vert u(0)\Vert _{1}\ \text{as }%
n\rightarrow \infty.
\end{equation}%
The functions $U_{n}(t),t\geq -h_{n}$ satisfy the $H^1$-norm equality \eqref{s2r10}), that is,
$$
\frac{d}{dt}\Vert U_{n}(t)\Vert _{1}^{2}+2r\Vert U_{n}(t)\Vert
_{1}^{2}=2(U_{n}(t),g)_{1}.
$$
We multiply this equation by $e^{2rt}$ and integrate in $t$ from $-h_{n}$ to $0$:
\begin{equation}
\Vert U_{n}(0)\Vert _{1}^{2}=\Vert U_{n}(h_{n})\Vert
_{1}^{2}e^{-2rh_{n}}+2\int_{-h_{n}}^{0}(U_{n}(t),g)_{1}e^{2rt}dt.
\label{s4r6}
\end{equation}%
We want to pass to the limit $n\rightarrow \infty $ in this equality. Recall
that $U_{n}(t)$ are uniformly bounded (see (\ref{s4r3})) and
$U_{n}(\cdot)\rightharpoondown U(\cdot )$ $\ast $-weakly in $L_{\mathrm{loc}}^{\infty }(%
\mathbb{R};\mathcal{H}^{1})$. Then we have that%
\begin{equation}
\int_{-h_{n}}^{0}(U_{n}(t),g)_{1}e^{2rt}dt\rightarrow \int_{-\infty
}^{0}(U(t),g)_{1}e^{2rt}dt\text{ as }n\rightarrow \infty .  \label{s4r7}
\end{equation}
Using (\ref{s4r7}) in (\ref{s4r6}), where $\Vert U_{n}(h_{n})\Vert _{1}^{2}$
are uniformly bounded, and the condition $u_{n}(0)=U_{n}(0)$ we obtain that
\begin{equation}
\underset{n\rightarrow \infty }{\lim }\Vert u_{n}(0)\Vert
_{1}^{2}=2\int_{-\infty }^{0}(U(t),g)_{1}e^{2rt}dt.  \label{s4r8}
\end{equation}%
But the $H^1$-norm equality \eqref{s2r10} also holds for the complete trajectory
$U(t),t\in \mathbb{R}$:
\begin{equation*}
\frac{d}{dt}\Vert U(t)\Vert _{1}^{2}+2r\Vert U(t)\Vert
_{1}^{2}=2(U(t),g)_{1}.
\end{equation*}%
We multiply this equation by $e^{2rt}$ and integrate from $-\infty $ to $0$:
\begin{equation}
\Vert u(0)\Vert _{1}^{2}=2\int_{-\infty }^{0}(U(t),g)_{1}e^{2rt}dt\quad
\text{(recall that }u(0)=U(0)\text{).}  \label{s4r9}
\end{equation}%
Hence, equalities (\ref{s4r8}) and (\ref{s4r9}) imply convergence~\eqref{s4r4},
which gives the strong convergence $u_{n}(0)\rightarrow u(0)\ $in $\mathcal{H%
}^{1}$ as $n\rightarrow \infty$ and proves~\eqref{strongH1} since
\begin{equation*}
\lim_{n\rightarrow \infty }\Vert u_{n}(0)-u(0)\Vert
_{1}^{2}=\lim_{n\rightarrow \infty }\left( \Vert u_{n}(0)\Vert
_{1}^{2}+\Vert u(0)\Vert _{1}^{2}-2(u_{n}(0),u(0))_{1}\right)=0.
\end{equation*}%

Next, the compactness of $\mathcal{A}$ follows from the above argument if
we take $B=\mathfrak{A}$.
Finally, the minimality property follows from the fact that
every compact attracting set must contain $\mathcal {K}(0)=\mathfrak{A}(0)=\mathcal{A}$
(see~\eqref{s3r1}). The proof is complete.
\end{proof}

We now prove that the weak trajectory attractor constructed in
Theorem \ref{s3the1} is in fact a strong trajectory attractor.

We consider the trajectory space $\mathcal{K}_{+}\subset C_{b}(\mathbb{R}_{+};%
\mathcal{H}^{1})$ and the weak trajectory attractor $\mathfrak{A\subset }%
\mathcal{K}_{+}$ of the system (\ref{s1r1}). In Theorem \ref{s3the1}, we
have established that $\mathfrak{A}$ exists and is bounded in
$C_{b}(\mathbb{R}_{+};\mathcal{H}^{1})$. On the trajectory space $\mathcal{K}_{+}$ we
consider now the local strong topology $\Theta _{+}^{\mathrm{loc,s}}$. By
 definition, a sequence $\{z_{n}\}\subset \mathcal{K}_{+}$ converges to
 $z$ in $\Theta _{+}^{\mathrm{loc,s}}$ if, for every $M>0,$ $z_{n}(t,x)\rightarrow z(t,x)$
 strongly in $C([0,M];\mathcal{H}^{1})$, that is
\begin{equation*}
\max_{t\in \lbrack 0,M]}\Vert z_{n}(t,\cdot)-z(t,\cdot)\Vert _{\mathcal{H}%
^{1}}\rightarrow 0\ (n\rightarrow \infty ).
\end{equation*}%
It is clear that the topology $\Theta _{+}^{\mathrm{loc,s}}$ is metrizable
and this metric space is complete.

\begin{theorem} \label{s4the1} Let $g\in\mathcal{H}^1$. Then, the
trajectory attractor $\mathfrak{A}$ is compact in $\Theta _{+}^{\mathrm{loc,s%
}}$ and, for any trajectory set $B\subset \mathcal{K}_{+},$ bounded in $%
C_{b}(\mathbb{R}_{+};\mathcal{H}^{1})$, the $h$-shifted  set $T(h)B$ tends to
$\mathfrak{A}$ as $h\rightarrow \infty $ in $\Theta _{+}^{\mathrm{loc,s}}$,
 that is, for every $M>0$
\begin{equation}
\mathrm{dist}_{C([0,M];\mathcal{H}^{1})}\left( T(h)B,\mathfrak{A}\right)
\rightarrow 0\ \text{\rm as}\ h\rightarrow +\infty.  \label{s4r1}
\end{equation}
\end{theorem}
\begin{proof}

To prove the theorem, we have to verify that
\begin{equation}
\Vert u_{n}(\cdot )\rightarrow u(\cdot )\Vert _{C([0,M];\mathcal{H}%
^{1})}\rightarrow 0  \label{s4r12}
\end{equation}%
for every fixed $M>0$, where the $u_{n}(\cdot )\in T(h_{n})B$ and
$u(\cdot )\in \mathfrak{A}$ were defined in the beginning of the proof of Theorem~\ref{T:glob}.
 We note that $B_{M}=\bigcup\nolimits_{s\in \lbrack 0,M]}T(s)B$ is
a bounded trajectory set. Then, by Theorem~\ref{T:glob}
\begin{equation*}
\mathrm{dist}_{\mathcal{H}^{1}}\left(B_{M}(h),\mathcal{A}\right)
\rightarrow 0\ \text {as}\  h\rightarrow +\infty.
\end{equation*}%
Now, since, $B_M(h)$ contains all the time  shifts of length $s\le M$
of the trajectories $u_n$, it follows that
\begin{equation}
\sup_{t\in \lbrack 0,M]}\mathrm{dist}_{\mathcal{H}^{1}}\left( u_{n}(t),%
\mathcal{A}\right) \rightarrow 0\ \text{as }n\rightarrow \infty .
\label{s4r13}
\end{equation}%
By Theorem~\ref{T:glob}, the set $\mathcal{A}$ is compact in $\mathcal{H}^{1}$.
Hence, for any $\varepsilon >0,$ there is a finite $\varepsilon $-net
$\{w_{j}\}_{j=1}^{N}$ for $\mathcal{A}$ in $\mathcal{H}^{1}$. By the density
of compactly supported functions in $H^1(\mathbb{R}^2)$ we can assume that
the supports of the $w_j$'s are contained in $B^0_R$, where
$R=R(N)$ is sufficiently large.  Let $P_{N}$
denote the orthogonal projection in $\mathcal{H}^{1}$ onto the
finite-dimensional linear span of $\{w_{j}\}_{j=1}^{N}$ and
let $Q_{N}=1-P_{N}.$
Then
\begin{equation}
\Vert Q_{N}u(t)\Vert _{1}\leq \varepsilon ,\ \forall t\in \lbrack 0,M],
\label{s4r14a}
\end{equation}
since $u\in \mathfrak{A}$. Due to (\ref{s4r13}),%
\begin{equation}
\underset{n\rightarrow \infty }{\lim \sup }\Vert Q_{N}u_{n}(t)\Vert _{1}\leq
\varepsilon ,\ \forall t\in \lbrack 0,M].  \label{s4r14b}
\end{equation}%
Recall that the sequence $u_{n}\rightharpoondown u\ \ast $-weakly
in $L^{\infty }(0,M;\mathcal{H}^{1}\mathcal{)}$ and $\partial
_{t}u_{n}\rightharpoondown \partial _{t}u\ \ast $-weakly in
$L^{\infty }(0,M; \mathcal{H}^{-1}\mathcal{)}$ for any $M>0.$ Then,
by Aubin--Lions--Simon theorem  $u_{n}\rightarrow u$ strongly in
$C([0,M];\mathcal{H}(B^0_r))$ for every $r<\infty$. We set $r=R(N)$
and obtain
$$
\sup_{t\in \lbrack 0,M]}\Vert P_{N}u_{n}(t)-P_{N}u(t)\Vert \rightarrow
0\ \text{as\ }n\rightarrow \infty,
$$
and, hence,
\begin{equation} \label{s4r15}
\sup_{t\in \lbrack 0,M]}\Vert P_{N}u_{n}(t)-P_{N}u(t)\Vert _{1}\rightarrow
0\ \text{as\ }n\rightarrow \infty ,
\end{equation}
because $P_{N}$ is a finite dimensional projection and the norms of $\mathcal{%
H}$ and $\mathcal{H}^{1}$ on $\mathrm{span}\{w_{j}\}$ are equivalent.
Finally, for an arbitrary $\varepsilon >0$ we chose the corresponding
projections $P_{N}$ and $Q_{N}$ for which we have (\ref{s4r14a}) and (\ref%
{s4r14b}). Then, thanks  to (\ref{s4r14b}) and (\ref{s4r15}), we find $%
n_{1}\in \mathbb{N}$ such that
$$
\sup_{t\in \lbrack 0,M]}\Vert Q_{N}u_{n}(t)\Vert _{1}\leq \varepsilon \quad
\text{and}\quad \sup_{t\in \lbrack 0,M]}\Vert P_{N}u_{n}(t)-P_{N}u(t)\Vert
_{1}\leq \varepsilon ,
$$
for all $ n\geq n_{1}$, which together with \eqref{s4r14a} and \eqref{s4r14b} implies that
$$
\aligned
&\sup_{t\in \lbrack 0,M]}\Vert u_{n}(t)-u(t)\Vert _{1}\le\\
& \sup_{t\in
\lbrack 0,M]}\left( \Vert P_{N}u_{n}(t)-P_{N}u(t)\Vert _{1}+\Vert
Q_{N}u_{n}(t)-Q_{N}u(t)\Vert _{1}\right) \leq \\
&\sup_{t\in \lbrack 0,M]}\Vert P_{N}u_{n}(t)-P_{N}u(t)\Vert
_{1}+
\sup_{t\in \lbrack 0,M]}\Vert Q_{N}u_{n}(t)\Vert _{1}+\\
&\qquad\qquad+\sup_{t\in
\lbrack 0,M]}\Vert Q_{N}u(t)\Vert _{1}\leq 3\varepsilon ,
\endaligned
$$
for all $n\geq n_{1}$. Therefore,\ we obtain (\ref{s4r12}) and
Theorem \ref{s4the1} is proved.
\end{proof}

Using the  strong convergence to $\mathfrak{A}$ and Corollary \ref{s2cor1}, we
obtain the following result.

\begin{corollary}
\label{s4cor1} For an arbitrary bounded trajectory set
$B\subset \mathcal{K}_{+}$ the corresponding time derivative set
\begin{equation*}
\partial _{t}B=\{\partial _{t}u,\ u\in B\}
\end{equation*}%
converges to $\partial _{t}\mathfrak{A}$ in the strong topology
$C_{loc}(\mathbb{R}_{+};\mathcal{H}^{-\varepsilon})$ for every
$\varepsilon>0$.
\end{corollary}

\section{Additional regularity and $W^{1,p}$-attractors \label{Sec5}}

In this section, we show that the  trajectory attractor
$\mathfrak{A}$ for the dissipative Euler system (\ref{s1r1}) constructed above is more regular
when the external force $g$ has some additional regularity.

We  define the Banach space $\mathcal{E}_p$ for $2\le p\le\infty$:
\begin{equation}\label{Ep}
\mathcal{E}_p:=\{u,\ \func{div}u=0,\ \|u\|_{\mathcal{E}_p}:=
\|u\|_{L^2}+\|\func{curl}u\|_{L^p}<\infty\}.
\end{equation}

We observe that for $\func{div}u=0$ and $1<p<\infty$ we have the equivalence of the norms
$$
\|\nabla u\|_{L^p}\sim\|\func{curl} u\|_{L^p}.
$$
In fact,  since $|\func{curl}u(x)|\leq |\nabla u(x)|$ point-wise,
it remains   to show that $\|\nabla u\|_{L^p}\le
c(p)\|\func{curl}u\|_{L^p}$. The proof of this inequality (along
with the information on the constant $c(p)$) is contained in
Lemma~\ref{s5cor1}.

Therefore, in view of the Gagliardo--Nirenberg  inequality
$$
\|u\|_{L^p(\mathbb{R}^2)}\le
c_1(p)\|u\|_{L^2(\mathbb{R}^2)}^{\alpha}\|\nabla u\|_{L^p(\mathbb{R}^2)}^{1-\alpha},
\quad\alpha=1/(p-1),\quad 2\le p,
$$
the $\mathcal{E}_p$-norm dominates the $W^{1,p}$-norm for $2\le p<\infty$:
\begin{equation}\label{Epdom}
\|u\|_{L^p(\mathbb{R}^2)}+\|\nabla u\|_{L^p(\mathbb{R}^2)}\le
c_2(p)\left(\|u\|_{L^2(\mathbb{R}^2)}+\|\func{curl} u\|_{L^p(\mathbb{R}^2)}\right).
\end{equation}

We need the
following theorem, which similarly to Theorem \ref{s2the1} essentially uses the
results in~\cite{DiP-Lio}.

\begin{theorem}
\label{s5the1} Let  $g\in \mathcal{H}^{1}\cap \mathcal{E}_p$ and
let $\omega _{0}\in L^{2}\cap L^p$  for $p\ge 2.$ Then the weak
solution~$\omega$ of the transport equation \eqref{s2r8} with
initial condition $\omega(0)=\omega _{0}$ and a given vector field
$u(t,x)\in L^{\infty }(0,T;\mathcal{H}^{1})$ satisfies the equation
\begin{equation}
\frac{1}{p}\frac{d}{dt}\|\omega (t)\|_{L^{p}}^{p}+r\|\omega(t)\| _{L^{p}}^{p}=
\int_{\mathbb{R}^{2}}\omega (t,x)\cdot |\omega
(t,x)|^{p-2}\func{curl}g(x)dx,  \label{s5r00}
\end{equation}%
and in addition to $\omega \in C_{b}(\mathbb{R}_{+};L^{2})$ we have
$\omega \in  C_{b}(\mathbb{R}_{+};L^{p})$.
\end{theorem}
\begin{proof} As before
the assertion of this theorem follows from the proof of
\cite[Theorem II.2 and equation (26)]{DiP-Lio}. The hypotheses of
the cited theorem are satisfied since our vector field $u$ belongs
to $L^{\infty}(0,T;\mathcal{H}^1)$ and therefore $u\in
L^{\infty}(0,T;W^{1,q}_\mathrm{loc}(\mathbb{R}^2)) \subset
L^{1}(0,T;W^{1,q}_\mathrm{loc}(\mathbb{R}^2))$ for all
$q=p'=p/(p-1)\le2$, as required in \cite{DiP-Lio}.
\end{proof}

\begin{theorem}\label{s5pro1}
Let $g\in \mathcal{H}^{1}\cap \mathcal{E}_p$ for some $2\le p<\infty$. Then the trajectory
attractor $\mathfrak{A}$ of
 system \eqref{s1r1} belongs to $C_{b}(\mathbb{R}_{+};\mathcal{E}_p)$, and
for every $u\in \mathfrak{A}$ the following estimate holds
\begin{equation} \label{s5r0}
\Vert u\Vert _{C_{b}(\mathbb{R}_{+};\mathcal{E}_p)}\leq \frac{1}{r}\Vert
g\Vert _{\mathcal{E}_p}.
\end{equation}
\end{theorem}

\begin{proof}
Let $u\in \mathfrak{A}$. According to (\ref{s3r1}), there is a bounded
complete trajectory $U\in \mathcal{K}$ such that $\Pi _{+}U=u.$ The
corresponding vorticity function $\omega (t,x)=\func{curl}U(t,x)$ satisfies
the equation%
\begin{equation*}
\partial _{t}\omega +(U,\nabla )\omega +r\omega =\func{curl}g(x),\ t\in
\mathbb{R}.  \label{s5r1}
\end{equation*}%
We consider the following Cauchy problem for $\tau \geq 0$:
\begin{equation*}\label{s5r2}
\aligned
\partial _{t}W^{\tau }+(U,\nabla )W^{\tau }+rW^{\tau }=\func{curl}g(x),
\quad W^{\tau }|_{t=-\tau }=0.
\endaligned
\end{equation*}
Since $\func{curl}g(x)\in L^{2}\cap L^{p}$, it follows from Theorem~\ref{s5the1}
that  this linear transport problem
 has a unique weak solution
$W^{\tau }\in C_{b}([-\tau,\infty );L^{2})\cap C_{b}([-\tau ,\infty );L^{p})$
satisfying for $t\ge \tau$  the following equation:
$$
\frac{1}{p}\frac{d}{dt}\Vert W^\tau(t)\Vert _{L^{p}}^{p}+r\Vert W^\tau(t)\Vert _{L^{p}}^{p}=
\int_{\mathbb{R}^{2}}W^\tau(t,x)\cdot |W^\tau(t,x)|^{p-2}\func{curl}g(x)dx.
$$
Using  H\"older's inequality  we obtain
\begin{equation*}
\frac{1}{p}\frac{d}{dt}\Vert W^\tau\Vert _{L^{p}}^{p}+r\Vert W^\tau\Vert
_{L^{p}}^{p}\leq \Vert \func{curl}g\Vert _{L^{p}}\Vert W^\tau\Vert _{L^{p}}^{p-1}.
\label{s5r5}
\end{equation*}%
Using  Young's inequality with parameter $\delta >0$:
$$ab\leq {\delta ^{-\frac{p}{q}}a^{p}}/p+{\delta b^{q}}/{q}$$ with
$a=\|\func{curl}g(x)\|_{L^{p}}$, $b=\|W^\tau\|_{L^{p}}^{p-1}$,  and $\delta=r$,
we obtain
\begin{equation*}
\frac{d}{dt}\Vert W^\tau\Vert _{L^{p}}^{p}+pr\Vert W^\tau\Vert _{L^{p}}^{p}\leq
r^{-(p-1)}\Vert \func{curl}g(x)\Vert _{L^{p}}^{p}+(p-1)r\Vert W^\tau\Vert
_{L^{p}}^{p},
\end{equation*}%
that is%
\begin{equation*}
\frac{d}{dt}\Vert W^\tau\Vert _{L^{p}}^{p}+r\Vert W^\tau\Vert _{L^{p}}^{p}\leq
r^{-(p-1)}\Vert \func{curl}g(x)\Vert _{L^{p}}^{p}.  \label{s5r6}
\end{equation*}%
Multiplying  both sides by $e^{rt}$
and integrating  over $[-\tau,t]$ we obtain the inequality for $t\ge -\tau$
$$
\Vert W^\tau(t)\Vert _{L^{p}}^{p}\leq \Vert W^\tau(-\tau )\Vert
_{L^{p}}^{p}e^{-r(t+\tau )}+(1-e^{-r(t+\tau )})r^{-p}\Vert \func{curl}g\Vert
_{L^{p}}^{p}.
$$
Since $W^\tau(-\tau )= 0$, this gives
\begin{equation}
\Vert W^\tau(t)\Vert _{L^{p}}\leq r^{-1}\Vert \func{curl}g\Vert
_{L^{p}},\quad t\ge -\tau. \label{s5r7}
\end{equation}%

Consider now the function $V^{\tau }(x,t)=\omega (t,x)-W^{\tau }(t,x)$,
which is the solution of the following linear homogeneous  transport problem%
$$
\partial _{t}V^{\tau }+(U,\nabla )V^{\tau }+rV^{\tau }=0,\quad
V^{\tau }|_{t=-\tau } =U(-\tau ),\quad  t\geq -\tau.
$$
>From \cite{DiP-Lio} (see also Theorem \ref{s2the1}), it follows that this
problem   has a unique solution satisfying the equation
\begin{equation*}
\frac{1}{2}\frac{d}{dt}\Vert V^{\tau }(t)\Vert ^{2}+r\Vert V^{\tau }(t)\Vert
^{2}=0,\ t\geq -\tau.
\end{equation*}%
Since  $U\in C_{b}(\mathbb{R};\mathcal{H}^{1})$, we have
\begin{equation}\label{to0}
\Vert V^{\tau }(t)\Vert ^{2}\leq \Vert V^{\tau }(-\tau )\Vert
^{2}e^{-2r(t+\tau )}=\Vert U^{\tau }(-\tau )\Vert ^{2}e^{-2r(t+\tau
)}\rightarrow 0,
\end{equation}
as $\tau\to\infty$.
We see from~\eqref{s5r7} and \eqref{to0} that for a fixed $t$ we have
$$
\aligned
&\omega-V^\tau\to \omega'\quad  \text{weakly in}\ L^p,\\
&\omega-V^\tau\to \omega\quad  \text{strongly in}\ L^2,
\endaligned
$$
as $\tau\to\infty$, where $\omega'\in L^p$ and
$\|\omega'\| _{L^{p}}\le r^{-1}\|\func{curl}g\|_{L^p}$.

Using  here a simple fact that if
$a_n\to a$ weakly in $L^2$ and $a_n\to a'$ weakly in $L^p$,
then $a=a'\in L^2\cap L^p$, we obtain that $\omega (t)$ also belongs to $L^{p}$ and
for all $t\ge0$ satisfies
\begin{equation} \label{s5r10}
\Vert \omega (t)\Vert _{L^{p}}\leq r^{-1}\Vert \func{curl}g\Vert _{L^{p}}.
\end{equation}%
To derive (\ref{s5r0}) from (\ref{s5r10}), we recall that every $u\in\mathfrak{A}$
 satisfies as before the inequality
\begin{equation*}
\Vert u(t)\Vert _{L^{2}}\leq r^{-1}\Vert g\Vert _{L^{2}},\quad t\geq 0,
\end{equation*}%
(see (\ref{s1r17})) and that the norm in $\mathcal{E}_p$ is given by~\eqref{Ep}.
\end{proof}

The next result extends Theorem~\ref{s4the1} to the $W^{1,p}$-case
(more precisely, to the $\mathcal{E}_p$-case).

\begin{theorem}\label{T:Lp} Let $g\in \mathcal{H}^{1}\cap \mathcal{E}_p$,  $2\le p<\infty$.
Then the trajectory attractor $\mathfrak{A}$ is compact in
$C_{\mathrm{loc}}(\mathbb{R}_{+};\mathcal{H}^{1}\cap \mathcal{E}_p)$ and if a trajectory set
$B\subset \mathcal{K} _{+}$ is bounded in
$C_{b}(\mathbb{R}_{+};\mathcal{H}^{1}\cap \mathcal{E}_p)$, then
for every $M>0$
\begin{equation}
\mathrm{dist}_{C([0,M];\mathcal{H}^{1}\cap \mathcal{E}_p)}\left( T(h)B,\mathfrak{A}\right)
\rightarrow 0\quad\text{\rm as}\  h\to +\infty.  \label{s5r11}
\end{equation}
\end{theorem}

The proof the theorem  is similar to the proof of Theorem \ref{s4the1}
and is based on the proof of the corresponding result for the global attractor.

\begin{theorem}\label{T:globWp}
The set
$$
\mathcal{A}:=\mathfrak{A}(0)=\{u(0),\ u\in \mathfrak{A}\}\subset \mathcal{H}^{1}\cap \mathcal{E}_p
$$
is the global attractor  in the
norm of  $\mathcal{H}^{1}\cap \mathcal{E}_p$.
\end{theorem}
\begin{proof}
We use the notation introduced in the proof of Theorem~\ref{T:glob},
in which we have shown that
$u_{n}(0)\to u(0)$ strongly in $\mathcal{H}^{1}$ as $n\to\infty$.
Now we have, in addition, that
\begin{equation}\label{curlweak}
\func{curl}u_{n}(0)\to \func{curl}u(0)
\end{equation}
weakly in $L^p$, and we have to show that the convergence~\eqref{curlweak}
is, in fact, strong.

For the functions $U_{n}(t,x),t\geq -h_{n}$ from the proof of Theorem~\ref{T:glob}
we have  the  equality \eqref{s5r00}, more precisely, the functions  $\omega_n=\func{curl}U_n$ satisfy for $t\geq -h_n$
\begin{equation} \label{equalLp}
\frac{1}{p}\frac{d}{dt}\Vert \omega_n (t)\Vert _{L^{p}}^{p}+r\Vert \omega_n
(t)\Vert _{L^{p}}^{p}=\int_{\mathbb{R}^{2}}\omega_n (t,x)\cdot |\omega_n
(t,x)|^{p-2}\func{curl}g(x)dx,
\end{equation}%
Multiplying by $e^{2rpt}$ and integrating from $-h_n$ to $0$ we obtain:
$$
\| \omega_{n}(0)\|_{L^p}^p=\|\omega_{n}(h_n)\|_{L^p}^pe^{-2rph_{n}}+2p
\int_{-h_{n}}^{0}\int_{\mathbb{R}^2}\omega_n |\omega_n|^{p-2}\func{curl}g\,e^{2rt}dxdt.
$$
Since $\omega_n\to\omega$ $\ast$-weakly in $L^\infty_{\mathrm{loc}}(\mathbb{R},L^p)$,
it follows that
$\omega_n |\omega_n|^{p-2}\to\omega |\omega|^{p-2}$ $\ast$-weakly in
$L^\infty_{\mathrm{loc}}(\mathbb{R},L^{q})$,
$q=p/(p-1)$, and therefore we obtain
$$
\lim_{n\to\infty}\| \omega_{n}(0)\|_{L^p}^p=
2p\int_{-\infty}^{0}\int_{\mathbb{R}^2}\omega |\omega|^{p-2}\func{curl}g\,e^{2rt}dxdt.
$$
Multiplying by $e^{2rpt}$ and integrating from $-\infty$ to $0$
the equation~\eqref{equalLp} for $\omega=\func{curl}U$, where $U$ is the
 complete trajectory $U(t)$, $t\in\mathbb{R}$ we obtain
$$
\| \omega(0)\|_{L^p}^p=
2p\int_{-\infty}^{0}\int_{\mathbb{R}^2}\omega |\omega|^{p-2}\func{curl}g\,e^{2rt}dxdt,
$$
which gives that
\begin{equation}\label{normstonorm}
\lim_{n\to\infty}\| \omega_{n}(0)\|_{L^p}=\|\omega(0)\|_{L^p}.
\end{equation}
The weak convergence~\eqref{curlweak} and \eqref{normstonorm} give (are equivalent to)
the strong convergence
$$
\|\omega_n(0)-\omega(0)\|_{L^p}\to0,
$$
see, for instance, \cite[Theorem~II.37]{Riesz-Nagy}.
The proof is complete.

\end{proof}

\section{$L^\infty$-bound and uniqueness \label{Sec6}}

\begin{lemma}\label{Cor:infty}
Let $g\in\mathcal{H}^1\cap\mathcal{E}_\infty$. Then on the
attractor $\mathfrak{A}$
\begin{equation}\label{omegainfty}
\|\omega(t)\|_{L^\infty}\le\frac1r\|\func{curl}g\|_{L^\infty}.
\end{equation}
\end{lemma}

\begin{proof}
We have shown in Theorem~\ref{s5pro1} that
$$
\|\omega(t)\|_{L^p}\le\frac1r\|\func{curl}g\|_{L^p}
$$
for every $p<\infty$. It remains to let $p\to\infty$ in this estimate.

In fact, since $\func{curl}g\in L^2\cap L^\infty$, in view of the elementary inequality
\begin{equation}\label{mult}
\|f\|_{L^p}\le\|f\|_{L^2}^{2/p}\|f\|^{1-2/p}_{L^\infty},\quad p\ge2,
\end{equation}
we find that
$$
\|\omega(t)\|_{L^p}\le\frac1r\|\func{curl}g\|_{L^p}\le
\frac1r\|\func{curl}g\|_{L^2}^{2/p}\|\func{curl}g\|^{1-2/p}_{L^\infty},
$$
and
\begin{equation}\label{ptoinf}
\limsup_{p\to\infty}\|\omega(t)\|_{L^p}\le
\frac1r\|\func{curl}g\|_{L^\infty}.
\end{equation}
This gives~\eqref{omegainfty} because
$\|f\|_{L^\infty}=\lim_{p\to\infty}\,\|f\|_{L^p}$ if it is known
that $\|f\|_{L^p}\leq C$ for all $p\geq p_{0}$.
\end{proof}

\begin{lemma}\label{s5cor1}
Let $g\in \mathcal{H}^{1}\cap \mathcal{E}_{\infty }$. Then
the trajectory attractor $\mathfrak{A}$ belongs to the space $C_{b}(\mathbb{R%
}_{+};\mathcal{W}^{1,p})$ for any $2\le p<\infty$ and for every $ u\in \mathfrak{A}$
\begin{equation} \label{s5r14}
\|\nabla u\| _{C_{b}(\mathbb{R}_{+};L^{p})}\leq Cp
\|\func{curl} g\| _{L^{\infty }},
\end{equation}%
where the constant $C=c'/r$ is independent of  $p$ as $p\to\infty$.
\end{lemma}
\begin{proof}
We recall that the vector function $u$ is recovered from the
vorticity $\omega$ by the  Biot-Savart law
$$
u=\nabla^\perp \psi:=(-\partial_2\psi,\partial_1\psi),
$$
where
$$
\psi=\frac1{2\pi}\int_{\mathbb{R}^2}\ln|x-y|\omega(y)dy
$$
solves in $\mathbb{R}^2$ the equation
\begin{equation}\label{Yud}
\Delta\psi=\omega.
\end{equation}

Next, we clearly have
$$
|\nabla u(x)|<\sum_{i,j=1}^2\left|\frac{\partial^2\psi(x)}{\partial x_i\partial x_j}\right|.
$$
Therefore \eqref{s5r14} follows from the Yudovich elliptic estimate
for the equation \eqref{Yud}:
\begin{equation}\label{Yudest}
\left\|\frac{\partial^2\psi(x)}{\partial x_i\partial x_j}\right\|_{L^p({\mathbb{R}^n})}
\le c_n\frac{p^2}{p-1}\|\omega\|_{L^p({\mathbb{R}^n})},
\end{equation}
where $1<p<\infty$ (see \cite{Yud1}--\cite{Yud4}) and
\eqref{ptoinf}.
\end{proof}

Finally, we prove that if  $g\in \mathcal{H}^{1}\cap\mathcal{E}_{\infty }$,
 then the solutions of the dissipative 2D Euler
system lying on the trajectory attractor $\mathfrak{A}$ are unique.
The damping does not play a role here and the proof closely
follows the celebrated  Yudovich
uniqueness theorem for the  2D Euler system (see \cite{Yud1,
Yud2}).

\begin{theorem}
\label{s5the2}Let $g\in \mathcal{H}^{1}\cap \mathcal{E}_{\infty }$,
and let $u_{1}(\cdot),u_{2}(\cdot)$ be two weak solutions of the
damped 2D Euler system (\ref{s1r1}). Let the  initial vorticities
$\omega _{i}(0)=\func{curl}u_{i}(0)$ be bounded: $\omega
_{1}(0),\omega _{2}(0)\in L^{\infty}$. Then, $u_{1}(0)=u_{2}(0)$
implies $u_{1}(t)=u_{2}(t)$ for all $t\in [0,M]$.
\end{theorem}

\begin{proof} Let $u(t)=u_{1}(t)-u_{2}(t)$. We set $E(t)=\| u_{1}(t)-u_{2}(t)\| ^{2}$.
Since $u\in C([0,M];\mathcal{H})$,  we have  $E\in C([0,M])$.
Suppose that $E(t)>0$ for $t\in (0,\delta )$ for some $\delta $ (otherwise,
the theorem is proved). Then $u(t)$ solves the equation:
\begin{equation*}
\partial _{t}u+(u_{1},\nabla )u-(u,\nabla )u_{2}+ru+\nabla(p_1-p_2)=0.
\end{equation*}
 Multiplying this equation by $u(t)$ and integrating
over $\mathbb{R}^{2},$ we have%
\begin{equation*}
\frac{d}{dt}E(t)+2rE(t)=\int_{\mathbb{R}^{2}}((u,\nabla )u_{2},u)dx\leq
\Vert \nabla u_{2}(t)\Vert _{L^{p}}\Vert u\Vert _{L^{2q}}^{2},
\end{equation*}%
where $p\ge2$ is arbitrary and $1/p+1/q=1$.
In  view of \eqref{mult},
 $$
 \|u\|_{L^{2q}}\le\|u\|_{L^{2}}^{1-1/p}\|u\|_{L^\infty}^{1/p},
 $$
 where $u(t)$ is bounded  in $L^{\infty }$, since $W^{1,p_0}(\mathbb{R}^2)\subset C_b(\mathbb{R}^2)$ for a $p_0>2$.
 Dropping the term $2rE(t)$ on the left-hand side and using the key estimate~\eqref{s5r14}, we obtain
\begin{equation}
\frac{d}{dt}E(t)\leq C_{2}p[E(t)]^{1-1/p},  \label{s5r16}
\end{equation}%
where  $C_{2}=C_2(\|g\|_{1},\|\func{curl}g\|_{L^\infty})$ is independent of $p$
as $p\to\infty$.  We now
pick the exponent $p$ in (\ref{s5r16}) in an optimal way, namely, we set
\begin{equation*}
p=\ln \frac{K}{E(t)},
\end{equation*}%
where $K>1$ is sufficiently large to guarantee that $p>2$ for all
$t\in (0,\delta )$ (recall that $E(t)$ is a bounded function). We derive from
\eqref{s5r16} the inequality
\begin{equation*}
\frac{d}{dt}E(t)\leq C_{2}E(t)\ln \frac{K}{E(t)}\left( E(t)^{-1/\ln \frac{K}{%
E(t)}}\right) \leq C_{3}E(t)\ln \frac{K}{E(t)},
\end{equation*}%
where we used the inequality
$E^{-1/\ln \frac{K}{E}}\leq e$ which (since $\ln (K/E)>0$)  is equivalent to the assumed inequality
$K\ge1$. We finally obtain
\begin{equation*}
\frac{d}{dt}\ln \frac{E(t)}{K}\leq -C_{3}\ln \frac{E(t)}{K}.
\end{equation*}%
Integrating this differential inequality over $[\varepsilon ,t]$ we
arrive at%
$$
\ln \frac{E(t)}{K}\leq \ln \frac{E(\varepsilon )}{K}e^{-C_{3}(t-\varepsilon)},
$$
which gives
$$
 E(t)\leq K\left[ \frac{E(\varepsilon )}{K}\right] ^{e^{-C_{3}(t-\varepsilon )}}
$$
for all sufficiently small $\varepsilon >0$. Passing to the limit
$\varepsilon \to 0+$ ($E(\varepsilon )$ is continuous in $\varepsilon$
 and $E(0)=0$), we obtain $E(t)\equiv 0.$ The theorem is proved.
\end{proof}

\end{document}